\documentclass[12pt]{amsart}
\usepackage{amsmath,amsthm,amsfonts,amssymb}
\begin{document}
\newcommand{\A}{{\mathbb A}}
\newcommand{\B}{{\mathbb B}}
\newcommand{\C}{{\mathbb C}}
\newcommand{\N}{{\mathbb N}}
\newcommand{\Q}{{\mathbb Q}}
\newcommand{\Z}{{\mathbb Z}}
\renewcommand{\P}{{\mathbb P}}
\renewcommand{\O}{{\mathcal O}}
\newcommand{\cF}{{\mathcal F}}
\newcommand{\R}{{\mathbb R}}
\newcommand{\rc}{\subset}
\newcommand{\rank}{\mathop{rank}}
\newcommand{\trace}{\mathop{tr}}
\newcommand{\dimc}{\mathop{dim}_{\C}}
\newcommand{\Lie}{\mathop{Lie}}
\newcommand{\Spec}{\mathop{Spec}}
\newcommand{\Auto}{\mathop{{\rm Aut}_{\mathcal O}}}
\newcommand{\alg}[1]{{\mathbf #1}}
\newtheorem{lemma}{Lemma}[section]
\newtheorem*{definition}{Definition}
\newtheorem*{claim}{Claim}
\newtheorem{corollary}{Corollary}
\newtheorem*{Conjecture}{Conjecture}
\newtheorem*{SpecAss}{Special Assumptions}
\newtheorem{example}{Example}
\newtheorem*{remark}{Remark}
\newtheorem*{observation}{Observation}
\newtheorem*{fact}{Fact}
\newtheorem*{remarks}{Remarks}
\newtheorem{proposition}[lemma]{Proposition}
\newtheorem{theorem}[lemma]{Theorem}
\numberwithin{equation}{section}
\def\labelenumi{\rm(\roman{enumi})}
\title[Rational Connectedness and order of maps]{%
Rational connectedness and order of non-degenerate meromorphic maps from $\C^n$
}
\author {Fr\'ed\'eric Campana \& J\"org Winkelmann}
\begin{abstract}
We show that an $n$-dimensional compact K\"ahler manifold $X$ admitting a 
non-degenerate meromorphic map $f:\C^n\to X$ of order $\rho_f<2$ is rationally
connected.
\end{abstract}
\subjclass{}%
%

\address{%
Fr\'ed\'eric Campana\\
Institut Elie Cartan BP239\\
Universit\'e de Lorraine\\
54506. Vandoeuvre-les-nancy Cedex. France\\
et: Institut Universitaire de France.
}

\address{%
J\"org Winkelmann \\
Lehrstuhl Analysis II \\
Mathematisches Institut \\
NA 4/73\\
Ruhr-Universit\"at Bochum\\
44780 Bochum \\
Germany
}
\thanks{
{\em Acknowledgement.}
The second author was
supported by the
SFB/TR 12 ``Symmetries and Universality in mesoscopic systems''.
}
\maketitle
\section{Introduction/Summary}
The purpose of this paper is the following result.

\begin{theorem}
Let $X$ be a compact K\"ahler manifold of fixed dimension $n$. Let $f:\C^n\dasharrow X$ be a non-degenerate meromorphic map of order $\rho_f<2$ (see the next section for the definition of $\rho_f$).

Then $X$ is rationally connected, hence projective.
\end{theorem}

This result belongs to a series of similar statements relating the existence and growth 
of maps from
$\C^n$ to algebro-geometric properties of the target space $X$.

These statements are better expressed by introducing the following invariant $\rho(X)$, which suggests many questions, some of which are raised in the last section: 

\begin{definition} Let $X$ be an $n$-dimension connected compact complex manifold. 
Define $\rho(X):=\inf \{\rho_f\vert f:\Bbb C^m\dasharrow X, m\geq n\}$, $f$
non-degenerate.

It is understood that $\rho(X)\in [0,+\infty]$, and that $\rho_f=+\infty$ if there exists no non-degenerate meromorphic map $f:\Bbb C^n\to X$.
\end{definition}

The invariant $\rho$ is easily seen to be bimeromorphic, preserved by finite \'etale covers, and `increasing' (ie: $\rho(X)\geq \rho(Y)$ if there exists a dominant meromorphic map $g:X\to Y$).

It is obvious that $\rho(X)=0$ if $X$ is unirational.

Kobayashi and Ochiai proved that the existence of a non-degenerate map
from $\C^n$ to a projective manifold $X$ implies that $X$ is not
of general type (\cite{KO}). 
Thus $\rho (X)=+\infty$ if $X$ is of general type. 
It is proved in \cite{Ca04}, more generally, that $X$ is `special' 
if there exists a non-degenerate meromorphic map
$f:\C^n\dasharrow X$. In particular, $X$ must be special if
$\rho(X)<+\infty$.

If $X$ is K\"ahler and $K_X$ is pseudo-effective of numerical dimension $\nu\in \{0,1,\dots, n\}$, then $\rho(X)\geq 2(1-\frac{\nu}{n})^{-1}\geq 2$ (\cite{CP}). This implies the previous result of Kobayashi-Ochiai (since $X$ is of general type if and only if $\nu=n$). Using \cite{BDPP}, it also implies that if $X$ is projective, then $X$ is uniruled if $\rho(X)<2$

If $h^0(X,Sym^k(\Omega^p_X))\neq 0$, for some $k,p>0$, then $\rho(X)\geq 2$ (\cite{NW}). The K\"ahler condition is not required here. 
If, in addition, $X$ is assumed
to be a K\"ahler surface, then the condition
$h^0(X,Sym^k(\Omega^p_X))=0\ \forall p,k$ implies that $X$ is K\"ahler.
Thus, if $X$ is a K\"ahler surface and if $\rho(X)<2$, then $X$ 
must be rational (\cite{NW}). On the other hand, some Hopf (thus non-rational, non-K\"ahler) surfaces have $\rho(X)\leq 1$ (\cite{NW}).

The present result generalizes these two results from \cite{CP}, \cite{NW}, avoiding the deep methods of \cite{CP}. The estimate $\rho_f<2$ is optimal, 
since for an Abelian variety $A$ we have $\rho_\tau=2$
for the universal covering map $\tau:\C^g\to A$.

All these results provide lower bounds for $\rho(X)$ deduced from the geometry of $X$. Producing upper bounds for $\rho(X)$ (ie: the existence of non degenerate $f's$) turns out to be a completely open topic, except in the trivial case of $X$ unirational or a torus.

The case of rationally connected vs unirational manifolds when $n\geq 3$ is of great interest. For example: what is $\rho(X)$ if $X$ is a `general' smooth quartic in $\Bbb P^4$? If $\rho(X)>0$, then $X$ is not unirational. More generally: are there projective (necessarily rationally connected) $X$ such that $\rho(X)\in ]0,2[$? See the last section for some more questions.

\section{Characteristic function and order of a non-degenerate meromorphic map}
We start with some preparations.
Let $X$ be a
 compact complex manifold and let $f:\C^n\dasharrow X$
be a meromorphic map. $f$ is said to be 
(differentiably) ``non-degenerate''
if there is a point $p\in\C^n$ such that $f$ is holomorphic at $p$
and such that  $(Df)_p:T_p\C^n\to T_{f(p)}X$ is surjective.

Let $\alpha:=dd^c||z||^2$ on $\C^n$, and let $\omega$
be a positive $(1,1)$-form on $X$.
The characteristic function of $f$ is defined as:
\[
T_f(r;\omega)=\int^r_1 \frac{dt}{t^{2n-1}}\int_{B_t}(f^*\omega)\wedge\alpha^{n-1}.
\]
Here $B_t=\{z\in\C^n:|z||<t\}$.
Observe that the integral over $B_t$ 
is well-defined even if $f$ is only
meromorphic, not necessarily holomorphic. (To see this, note that locally
$\omega$ can be
dominated by a sum $\sum_i\alpha_i\wedge\bar\alpha_i$ where the
$\alpha_i$ are holomorphic $1$-forms. The holomorphy of the $\alpha_i$ implies
that $\sum_i\alpha_i\wedge\bar\alpha_i$ extends real-analytically through the indeterminacy
set of $f$.)

If $\omega$ and $\tilde\omega$ are any two positive $(1,1)$-forms
on $X$, then (by the compactness of $X$)
there are constants $C_1,C_1>0$ such that
\begin{equation}
C_1\omega <\tilde\omega< C_2\omega \label{compare-forms}
\end{equation}
and consequently
\[
C_1T_f(r,\omega) <T_f(r,\tilde\omega)< C_2T_f(r,\omega)\ \ \forall r>1.
\]

The ``order'' $\rho_f$ is defined as
\[
\rho_f=\limsup_{r\to\infty} \frac{\log T_f(r,\omega)}{\log r}
\]
where $\omega$ is a positive $(1,1)$-form on $X$.
Due to the inequalities \ref{compare-forms} this number
$\rho_f\in\R\cup\{+\infty\}$
is independent of the choice of the positive $(1,1)$-form $\omega$
used in the definition of $T_f(r,\omega)$.

The well-known
 Crofton's formula (see e.g.~\cite{G}) implies that $T_f(r)$ equals
the average of $T_{f|L}(r)$ taken over all
complex lines $L\subset \C^n$. This permits to extend many properties of
characteristic functions from the case of entire curves to the higher-dimensional domains.

For instance, let $\tau:X'\to X$ be a bimeromorphic holomorphic map between
compact complex manifolds.
Then $\tilde f=\tau^{-1}\circ f:\C^n\to X'$ is again a
non-degenerate meromorphic map, and $\rho_f=\rho_{\tilde f}$.

\label{?} Now let $\alpha:X\to Y$ be a dominant holomorphic map.
Then it follows directly from the definitions that $T_{\alpha\circ f}(r)
\le T_f(r)$ (if $\omega_X\ge \alpha^*(\omega_Y)$) and consequently $\rho_{\alpha \circ f}\le\rho_f$.

The order is easily seen to behave nicely with respect to products.
Let $X=X_1\times X_2$, let $\omega_i$ be positive $(1,1)$-forms in $X_i$
and let $f_i:\C^{n_1}\to X_i$ be non-degenerate meromorphic maps.
Define $\omega=\pi_1^*\omega_1+\pi_2^*\omega_2$, $n=n_1+n_2$ and
$f:\C^n\to X$ as $f(v_1,v_2)=(f_1(v_1),f_2(v_2))$.

Then $T_f(r)=T_{f_1}(r)+T_{f_2}(r)$. This implies
$\rho_f=\max\{\rho_{f_1},\rho_{f_2}\}$, because for any $t_1,t_2>0$ we
have
\[
\max_i\log t_i\le\log(t_1+t_2)\le \log(2\max_i t_i)=\log 2 +\log \max_i t_i.
\]

\section{Rational connectedness}

As usual, a compact complex manifold $X$
(usually K\"ahler) is called ``rationally connected''
(short: RC) if every two points can be linked by a (possibly singular) irreducible rational curve.
(Equivalently in the K\"ahler case: can be linked by a chain of rational curves).
Unirational manifolds are RC.
There are unirational threefolds which are not rational,
e.g.~smooth cubic hypersurfaces in $\P_4$.
However, it is not known whether there exist RC
manifolds which are not unirational, although it is expected that these should exist (the case of `general' quartics in $\P_4$ being one of the first open cases) .

Rationally connected compact K\"ahler manifolds are projective (since $h^{2,0}=0$, using Kodaira's criterion).

Let $X$ be a compact connected K\"ahler manifold.
Then there exists an ``almost holomorphic'' rational dominant map
$\rho:X\dasharrow Y$ (called the `` $RC$-reduction", or ``rational quotient"
\cite{C}, or the ``MRC-fibration"
\cite{KMM} if $X$ is projective), such that the fibers are
RC, and maximum with this property.

When $X$ is projective, it is known (by \cite{GHS}) that the base $Y$ is not uniruled. In fact, \cite{GHS} shows that if $f:X\dasharrow Y$ is a surjective meromorphic map with fibres and base $Y$ which are both RC, then $X$ is RC if it is compact K\"ahler (remark first that $h^{2,0}(X)=0$, and that $X$ is thus projective). The base $Y$ of the RC-reduction $\rho:X\dasharrow Y$  is not uniruled also when $X$ is compact K\"ahler.
Let indeed $r:Y\dasharrow Z$ be the RC-reduction of $Y$. The fibres
of $r\circ \rho:X\dasharrow Z$ are thus RC, and so $Y=Z$, which means that $Y$ is not uniruled.

 Due to \cite{BDPP} a projective manifold is uniruled if and only if $K_X$
is not pseudoeffective. Based on this, another recent criterion is the following
(see \cite{CDP}):

Let $X$ be a compact K\"ahler manifold.
Then $X$ is rationally connected if and only if there
is no pseudoeffective invertible subsheaf $F\subset\Omega^p_X$.
(for some $p\in\N$).

The proof consists in observing that 
$X$ is not RC precisely if its RC-reduction $\rho:X\dasharrow Y$ has $dim(Y)=p>0$. Define then $F:=\rho^*(K_Y)$, which is pseudo-effective since $Y$ is not uniruled.

It is conjectured that a compact K\"ahler (or equivalently: projective) manifold $X$ is RC if (and only if) there is no $\Q$-effective (instead of pseudo-effective) invertible subsheaf $F\subset\Omega^p_X$ (for some $p\in\N$). By means of the RC-reduction as above, this conjecture is equivalent to the `non-vanishing conjecture', claiming that if $K_X$ is pseudo-effective, it is $\Q$-effective.
 
\section{Pseudoeffective Line bundles}

A singular hermitian metric $h$ on a complex line bundle is
given in the form $|\ |_{h}=e^{-\phi}|\ |_s$ where $s$ is the standard
 metric for some local holomorphic trivialization
and $\phi$ is a $L^1_{loc}$-function.

The $L^1_{loc}$-condition ensures that the curvature $\Theta=
dd^c\log(e^{-\phi}) = -dd^c\phi$ makes sense (in the sense of currents)
and represents the Chern class of the line bundle.

A line bundle $L$ on a compact complex
manifold is called ``pseudoeffective'' iff  there is a singular
hermitian metric $h$ with
semipositive curvature $\Theta_h\ge 0$.
This condition means that the metric is locally given via a
weight function $e^{-\phi}$ with $\phi$ plurisubharmonic.
If $s$ is a holomorphic section in $F$, this implies
that $-\log||s||_h$ is plurisubharmonic.

\section{Measuring the derivative}
Let $V,W$ be complex vector spaces equipped with hermitian inner
products.
The norm $||F||$ of a complex linear map $F:V\to W$ is defined
by: $||F||^2=trace(F^*\circ F)$ where $F^*$ denotes the adjoint of $F$.
 If $A$ is the matrix describing $F$
with respect to orthonormal bases on $V$ and $W$, then
\[
||F||^2 = \sum_{i,j} |A_{ij}|^2.
\]

We continue with a local observation. Let $U,V$ be open subsets in $\C^n$,
let $\alpha = dd^c||z||^2 = \sum_ j i\cdot dz_j\wedge d\bar z_j$ be the
K\"ahler form for the euclidean metric and let $f:U\to V$ be a
holomorphic map.

Then
\[
f^*\alpha \wedge \alpha^{n-1} =\frac 1n ||Df||^2 \alpha^n
\]
where
\[
||Df||^2 =
trace\left((Df)^*\circ (Df)\right)
= \sum_{j,k}\left| \frac{\partial f_j}{\partial z_k} \right|^2
\]

\begin{proposition}\label{prop-A}
Let $U\subset\C^n$ be an open subset, let $X$ be a compact K\"ahler manifold
equipped with a hermitian metric $h$ and a positive $(1,1)$-form $\omega$ and
let $f:U\to X$ be a holomorphic map.

Then there is a constant $C>0$ such that
\[
f^*\omega\wedge\alpha^{n-1}\ge C ||Df||^2\alpha^n.
\]
Here $||Df||$ is calculated with respect to $h$.
\end{proposition}
\begin{proof}
We cover $X$ with finitely many open subsets $V_k$ such that each $V_k$
admits an embedding $j_k:V_k\to \C^n$ and each $V_k$ contains a relatively
compact open subset $W_k\subset V_k$ such that $X=\cup_kW_k$.
Let $h_k$ denote the hermitian metric on $V_k$ induced by the euclidean metric
via its embedding in $\C^n$. We choose $C_1>0$ such that $h_k\le C_1h$
everywhere on each $W_k$
and $C_2>0$ such that $\omega\ge C_2
j_k^*\alpha$ on each $W_k$.
Then the claim (with $C=C_1C_2$) follows from the preceding local
observation.
\end{proof}

\section{The result}
We show:
\begin{theorem}
Let $X$ be a compact K\"ahler manifold. Let $f:\C^n\dasharrow X$ be a non-degenerate meromorphic map of order $\rho_f<2$.

Then $X$ is projective and rationally connected.
\end{theorem}

\begin{proof}
Assume by contradiction that $X$ is not rationally connected.
There is then a
pseudoeffective invertible subsheaf $\cF\subset\Omega^p_X$
(for some $p\in\N$).
We fix a hermitian metric $h$ on $X$. The hermitian metric
on $T_X$ induces a hermitian metric on $\Omega_X^p$,
by abuse of language also denoted by $h$.
The injection of sheaves $\cF\hookrightarrow\Omega^p_X$ corresponds
to a non-zero vector bundle homomorphism $\xi_0:F\to \Omega^p_X$, where
$F$ is a pseudoeffective line bundle.
Let $g$ denote a smooth hermitian metric on $F$
with $g\le h|_F$ and let $g_0$ denote
a singular hermitian metric on $F$ such that $\Theta_{g_0}\ge 0$, i.e,
with positive curvature current.
Then there is an upper semicontinuous function $\phi\to\R$ such that
$g_0=e^{-\phi}g$. Since $\phi$ is upper semicontinuous, and
$X$ is compact, it is bounded from above: $M:=\sup_{x\in X}\phi(x)<\infty$.

The meromorphic map $f:\C^n\to X$ and the vector bundle
homomorphism $\xi_0:F\to \Omega^p_X$ induce vector bundle homomorphisms
\[
f^*F\stackrel\xi\longrightarrow f^*\Omega^p_X
\stackrel{Df^*}\longrightarrow \Omega^p_{\C^n}.
\]
on $\C^n\setminus I(f)$ (with $I(f)$ denoting the indeterminacy set
of the meromorphic map $f$.)
We are interested in a lower bound for $||Df||=||(Df)^*||$, calculated
with respect to 
the metric induced by $h$ resp.~the euclidean metric on
$f^*\Omega^p_X $ resp.~$\Omega^p_{\C^n}$. 
Let $\beta:=(Df^*)\circ\xi$.
Since we assumed
$g\le h|_F$, we have $||\beta||_g\le||Df||$ where
$||\beta||_g\le||Df||$ denotes the norm with respect to $g$ on $F$
and the euclidean metric on $\Omega^p_{\C^n}$.
Let $||\beta||_{g_0}$ denote the norm with respect to $g_0$ on $F$.

By using the standard trivialization of $\Omega^p_{\C^n}$, the bundle
morphism $\beta:f^F\to\Omega^p_{\C^n}$ corresponds to a vector valued
section on the dual bundle $f^*F*$.
Now $\Theta_{g_0}\ge 0$ implies that
$\log||s||_{g_0}$ is plurisubharmonic for every 
holomorphic section $s$ in $f^*F^*$.

Hence $\log||\beta||_{g_0}$ is a plurisubharmonic function on
$\C^n\setminus I(f)$ where $I(f)$ denotes the indeterminacy set
of the meromorphic map $f$.
Plurisubharmonic functions
extend through closed analytic subsets of codimension at least two.
Hence $\log||\beta||_{g_0}$ extends to a plurisubharmonic function
$\zeta_0$ defined on the whole $\C^n$.

We observe that $||\beta||_{g_0}=e^{\phi}||\beta||_{g}$, because
$g_0=e^{-\phi}g$.

By the definition of the constant $M$, we have
$||\beta||_{g_0} \le e^M||\beta||_{g}$,
implying 
\[
\zeta_0- M \le \log||\beta||_{g} \le \log ||Df||
\]

Define $\zeta=\exp(\zeta_0-M)$.
The plurisubharmonicity of $\zeta_0$
implies the plurisubharmonicity of $\zeta$.
Thus we obtain
the existence of a plurisubharmonic function $\zeta$ on $\C^n$
such that $||Df||\ge\zeta$.

Using proposition~\ref{prop-A} we may deduce that
\[
f^*\omega\wedge\alpha^{n-1}\ge\zeta\alpha^n.
\]
It follows that
\[
T_f(r;\omega)\ge
\int^r_1 \frac{dt}{t^{2n-1}}\int_{B_t}\zeta\alpha^n
\]
Since moving the origin, i.e., replacing $f$ by $f\circ\tau$ where $\tau$
denotes a translation, does not affect the order $\rho_f$, we may assume that
$c=\zeta(0,\ldots,0)>0$. Using the sub-mean value property of plurisubharmonic
functions it follows that
\[
T_f(r;\omega)\ge
c\int^r_1 \frac{dt}{t^{2n-1}}\int_{B_t}\alpha^n
=c\nu \int^r_1 \frac{dt}{t^{2n-1}}t^{2n}
=\frac{c\nu}{2} r^2+O(1)
\]
where $\nu$ denote the volume of the unit ball.
Therefore
\[
\rho_f=\limsup\frac{\log T_f(r)}{\log r}
\ge \limsup\frac{\log (r^2)}{\log r}=2.
\]
\end{proof}

\section{Non-K\"ahler manifolds}
Our result is not valid for non-K\"ahler manifolds. In fact,
there are Hopf surfaces $X$ admitting a non-degenerate holomorphic
map $f:\C^n\to X$ of order $\rho_f=1$
(see \cite{NW}). Of course, these Hopf surfaces
are non-K\"ahler and not rationally connected; they do not contain
any rational curve at all.

More precisely, in \cite{NW} the following is proved:
\begin{theorem}
Let $\lambda\in\C$ with $|\lambda|>1$ and let
$\sim$ denote the equivalence relation on $\C^2\setminus\{(0,0)\}$
given by 
\[
v\sim w \quad \iff \quad \exists k: v=\lambda^k w.
\]
Let $X=\C^2\setminus\{(0,0)\}/\sim$ and let $f:\C^2\to X$ be the map induced
by $(z,w)\mapsto (z,1+zw)$.

Then $\rho_f=1$.
\end{theorem}

This result has been generalized by T.~Amemiya to the class of Hopf
surfaces defined by an equivalence $(z,w)\sim(\lambda^kz,\mu^kw)$ ($k\in\Z$)
where $\lambda$ may be different from $\mu$ (but $|\lambda|,|\mu|>1$).

\section{Questions}

The following questions are not expected to have necessarily positive answers.

Let $X$ be an $n$-dimensional compact K\"ahler manifold, $f:\C^m\dasharrow X$ meromorphic non-degenerate. 

1. Is $X$ unirational if $\rho_f=0$? Is $X$ unirational if $\rho(X)=0$?

(It should be remarked that $\rho_f=0$ for every algebraic map,
but the condition $\rho_f$ is substantially weaker than algebraicity, as seen by appropriate power series in one variable.)

2. If there exists such an $f:\C^m\dasharrow X$, can it be chosen so that $\rho_f<+\infty$? In other words: if there exists an $f$ as above, is $\rho(X)<+\infty$?

3. If $\rho(X)<+\infty$, does there exist some $f:\C^m\dasharrow X$ with $\rho_f=\rho(X)$?

4. If $X$ is rationally connected, does there exist a non-degenerate meromorphic map $f: \C^n\dasharrow X$? Is then $\rho(X)<+\infty$?

5. If $X$ is RC, and if there exists a non-degenerate $f:\C^m\dasharrow X$, is $X$ unirational? (ie: can $f$ be chosen algebraic?). A positive answer would imply that there exists no $X$ with $\rho(X)\in ]0,2[$. 

The questions 3 and 4 were raised for $n=3$ in \cite{Ca04}, question 9.5.

6. Is the estimate $\rho(X)\geq 2(1-\frac{\nu}{n})^{-1}$ in \cite{CP} optimal if $K_X$ is pseudoeffective with $\nu(X)=:\nu$? In other words: does there exists $X_n$ with $\nu(X)=\nu$ (or better: with $\kappa(X)=\nu$) and with $\rho(X)\geq 2(1-\frac{\nu}{n})^{-1}$ for any $n>0$ and $\nu\in \{0,1,\dots,n\}$?

\end{document}